\documentclass[a4paper,12pt]{article}
\usepackage{cmap}                        
\usepackage[cp1251]{inputenc}            
\usepackage[english]{babel}
\usepackage[left=2cm,right=2cm,top=2cm,bottom=2cm]{geometry} 
\usepackage{amssymb}
\usepackage{amsmath, amsthm}
\theoremstyle{plain}
\newtheorem{theorem}{Theorem}
\newtheorem{lemma}{Lemma}

\newtheorem{corollary}{Corollary}

\theoremstyle{definition}
\newtheorem{example}{Example}
\newtheorem{remark}{Remark}

\newtheorem{pr}{Problem}

\begin{document}

\begin{center}\Large
\textbf{On the quasi-$\mathfrak{F}$-hypercenter of a finite group}
\normalsize

\smallskip
V.\,I. Murashka

 \{mvimath@yandex.ru\}

 Francisk Skorina Gomel State University, Gomel\end{center}

\textbf{Abstract.} In   this paper some properties of the $\mathfrak{F}^*$-hypercenter of a finite group are studied where $\mathfrak{F}^*$ is the class of all finite quasi-$\mathfrak{F}$-groups for a hereditary  saturated formation $\mathfrak{F}$ of finite groups. In particular, it is   shown that the quasinilpotent hypercenter of a finite group coincides with the intersection of all maximal quasinilpotent subgroups.


 \textbf{Keywords.} Finite groups; quasinilpotent groups; quasi-$\mathfrak{F}$-groups; hereditary saturated formation; hypercenter of a group.

\textbf{AMS}(2010). 20D25,  20F17,   20F19.

\section*{Introduction and the results}

All considered groups are finite. In \cite{h1} R. Baer showed that  the hypercenter $\mathrm{Z}_\infty(G)$ of a  group $G$ coincides with the intersection of all maximal nilpotent subgroups of $G$.
The concept of hypercenter was extended on classes of groups (see \cite[p. 127--128]{s6} or \cite[1, 2.2]{s5}). Let $\mathfrak{X}$ be a class of groups. A chief factor $H/K$ of a group $G$ is called   $\mathfrak{X}$-central if $(H/K)\leftthreetimes G/C_G(H/K)\in\mathfrak{X}$ otherwise it is called $\mathfrak{X}$-eccentric. A normal subgroup $N$ of $G$ is said to be $\mathfrak{X}$-hypercentral in $G$ if $N=1$ or $N\neq 1$ and every chief factor of $G$ below $N$ is $\mathfrak{X}$-central. The $\mathfrak{X}$-hypercenter $\mathrm{Z}_\mathfrak{X}(G)$ is the product of all normal   $\mathfrak{X}$-hypercentral subgroups of $G$. So if $\mathfrak{X}=\mathfrak{N}$ is the class of all nilpotent groups then $\mathrm{Z}_\infty(G)=\mathrm{Z}_\mathfrak{N}(G)$ for every group $G$.



   L.\,A. Shemetkov possed the following problem on the Gomel Algebraic Seminar in 1995: \textquoteleft\textquoteleft Describe all hereditary saturated formations $\mathfrak{F}$  such that  $\mathrm{Int}_\mathfrak{F}(G)=\mathrm{Z}_\mathfrak{F}(G)$ holds for every group $G$\textquoteright\textquoteright. This problem was solved by A.\,N. Skiba in \cite{h4}.
Let $F$ be the canonical local definition of a nonempty local formation $\mathfrak{F}$. Then $\mathfrak{F}$ is said to satisfy the boundary condition \cite{h4} if  $\mathfrak{F}$ contains every  group $G$  whose all maximal subgroups belong to $F(p)$ for some prime $p$.

\begin{theorem}[{\cite[Theorem A]{h4}}]\label{t0} Let $\mathfrak{F}$ be a hereditary saturated formation.  Then  $\mathrm{Int}_\mathfrak{F}(G)=\mathrm{Z}_\mathfrak{F}(G)$ holds for every group $G$  if and only if $\mathfrak{F}$ satisfies the boundary condition.  \end{theorem}


 The natural generalization of a saturated formation is a solubly saturated formation. Recall that a formation $\mathfrak{F}$ is called solubly saturated if from $G/\Phi(G_\mathfrak{S})\in\mathfrak{F}$ it follows that $G\in\mathfrak{F}$ where $G_\mathfrak{S}$ is the soluble radical of a group $G$. Every saturated formation is solubly saturated. The class $\mathfrak{N}^*$ of all quasinilpotent groups is the example of a solubly saturated and not saturated formation. Hence the following problem seems natural:

       \begin{pr}\label{prob1} Describe all  normally hereditary solubly saturated  formations $\mathfrak{F}$ such that $\mathrm{Int}_\mathfrak{F}(G)=\mathrm{Z}_\mathfrak{F}(G)$ holds for every group $G$.   \end{pr}

\begin{remark}    The solution of Problem \ref{prob1} when $\mathfrak{F}$ is a normally hereditary  saturated  formation is not known at the time of writing. \end{remark}

       Recall \cite[3.4.8]{s8} that every solubly saturated  formation $\mathfrak{F}$ contains the greatest saturated subformation $\mathfrak{F}_l$ with respect to inclusion. A.\,F. Vasil'ev suggested the following problem:

\begin{pr}[A.\,F. Vasil'ev]\label{Vas}
$(1)$ Let $\mathfrak{H}$ be a saturated formation. Assume that  $\mathrm{Int}_\mathfrak{H}(G)=\mathrm{Z}_\mathfrak{H}(G)$ holds for every group $G$. \!Describe all  normally hereditary solubly saturated formations $\mathfrak{F}$ with $\mathfrak{F}_l=\mathfrak{H}$\! such that \!$\mathrm{Int}_\mathfrak{F}(G)=\mathrm{Z}_\mathfrak{F}(G)$ holds for every group $G$.

$(2)$ Let $\mathfrak{F}$ be a normally hereditary solubly saturated formation. Assume that  $\mathrm{Int}_\mathfrak{F}(G)=\mathrm{Z}_\mathfrak{F}(G)$ holds for every group $G$. Does $\mathrm{Int}_{\mathfrak{F}_l}(G)=\mathrm{Z}_{\mathfrak{F}_l}(G)$ hold for every group $G$?
 \end{pr}

Let $\mathfrak{F}$ be a normally hereditary solubly saturated formation. The following example shows that if $\mathrm{Int}_{\mathfrak{F}_l}(G)=\mathrm{Z}_{\mathfrak{F}_l}(G)$ holds for every group $G$ then it is not necessary that  $\mathrm{Int}_\mathfrak{F}(G)=\mathrm{Z}_\mathfrak{F}(G)$ holds for every group $G$.

\begin{example} Recall that $\mathfrak{N}_{ca}$ is the class of groups whose abelian chief factors are central and non-abelin chief factors are simple groups. According to \cite{Ved} $\mathfrak{N}_{ca}$ is a normally hereditary solubly saturated formation. As follows from  \cite[3.4.5]{s8}   the greatest saturated subformation with respect to inclusion of $\mathfrak{N}_{ca}$  is $\mathfrak{N}$. Recall that  $\mathrm{Int}_\mathfrak{N}(G)=\mathrm{Z}_{\mathfrak{N}}(G)$  holds for every group $G$.

Let $G\simeq D_4(2)$ be a Chevalley orthogonal group and $H$ be a $\mathfrak{N}_{ca}$-maximal subgroup of $Aut(G)$. We may assume that $G\simeq Inn(G)$ is a normal subgroup of $Aut(G)$.  Since $G$ is simple and $HG/G\in \mathfrak{N}_{ca}$,   $HG\in\mathfrak{N}_{ca}$ by the definition of $\mathfrak{N}_{ca}$. Hence $HG=H$. So $G$ lies in the intersection of all $\mathfrak{N}_{ca}$-maximal subgroups of $Aut(G)$. From \cite{jm} it follows that $Aut(G)/C_{Aut(G)}(G)\simeq Aut(G)$.  If $G\leq\mathrm{Z}_{\mathfrak{N}_{ca}}(Aut(G))$ then  $G\leftthreetimes (Aut(G)/C_{Aut(G)}(G))\simeq G\leftthreetimes Aut(G)\in\mathfrak{N}_{ca}$. Note that $Out(G)\simeq S_3\not\in\mathfrak{N}_{ca}$ is the quotient group of $G\leftthreetimes Aut(G)$. Therefore $G\leftthreetimes Aut(G)\not\in\mathfrak{N}_{ca}$. Thus  $G\not\leq\mathrm{Z}_{\mathfrak{N}_{ca}}(Aut(G))$ and $\mathrm{Int}_{\mathfrak{N}_{ca}}(Aut(G))\neq\mathrm{Z}_{\mathfrak{N}_{ca}}(Aut(G))$.  \end{example}


In \cite{sk2, sk1} W. Guo and A.\,N. Skiba introduced the concept of quasi-$\mathfrak{F}$-group for a saturated formation $\mathfrak{F}$. Recall that a group $G$ is called quasi-$\mathfrak{F}$-group if for every $\mathfrak{F}$-eccentric chief factor $H/K$ and every $x\in G$, $x$ induces inner automorphism on $H/K$. We use $\mathfrak{F}^*$ to denote the class of all quasi-$\mathfrak{F}$-groups. If $\mathfrak{N}\subseteq\mathfrak{F}$ is a normally hereditary saturated formation then $\mathfrak{F}^*$  is a normally hereditary solubly saturated formation by \cite[Theorem 2.6]{sk2}.

\begin{theorem}\label{t1} Let $\mathfrak{F}$ be a hereditary saturated formation containing all nilpotent groups. Then   $\mathrm{Int}_\mathfrak{F}(G)=\mathrm{Z}_\mathfrak{F}(G)$ holds for every group $G$  if and only if  $\mathrm{Int}_{\mathfrak{F}^*}(G)=\mathrm{Z}_{\mathfrak{F}^*}(G)$ holds for every group $G$.    \end{theorem}

\begin{corollary} The intersection of all maximal quasinilpotent subgroups of a group $G$ is $\mathrm{Z}_{\mathfrak{N}^*}(G)$.  \end{corollary}

\begin{remark} Let $\mathfrak{N}\subseteq\mathfrak{F}$ be a hereditary saturated formation. As follows from  \cite[3.4.5]{s8} and \cite[1.3, 3.6]{s5} the greatest saturated subformation with respect to inclusion of $\mathfrak{F}^*$  is $\mathfrak{F}$. \end{remark}

\begin{remark} Note that every $\mathfrak{N}$-central chief factor is central.  From $(4)$ and $(5)$ of  the proof of Theorem \ref{t1} it follows that $\mathrm{Z}_{\mathfrak{N}^*}(G)$ is the greatest normal subgroup of $G$ such that every element of $G$ induces an inner automorphism on every chief factor of $G$  below  $\mathrm{Z}_{\mathfrak{N}^*}(G)$.  \end{remark}


\section{Preliminaries}

The notation and terminology agree with the books \cite{s8, s5}. We refer the reader to these
books for the results on formations. Recall that $\pi(G)$ is the set of all prime divisors of a group $G$,  $\pi(\mathfrak{F})=\underset{G\in\mathfrak{F}}\cup\pi(G)$ and $\mathbb{P}$ is the set of all primes.

Recall that if $\mathfrak{F}$ is a saturated formation, there exists a unique formation function $F$, defining $\mathfrak{F}$, which is integrated ($F(p)\subseteq \mathfrak{F}$ for every $p\in  \mathbb{P}$, the set of all primes)
and full ($F(p) = \mathfrak{N}_p F(p)$ for every $p\in\mathbb{P}$) \cite[IV, 3.8]{s8}. $F$ is called the canonical local definition of $\mathfrak{F}$.

\begin{lemma}[{\cite[1, 1.15]{s5}}]\label{l1} Let $\mathfrak{F}$ be a saturated formation, $F$ be its canonical local definition  and $G$ be a group. Then a chief factor $H/K$ of  $G$ is $\mathfrak{F}$-central if and only if $G/C_G(H/K)\in F(p)$ for all $p\in\pi(H/K)$.
\end{lemma}

The following lemma directly  follows from   \cite[X, 13.16(a)]{Hup}.

\begin{lemma}\label{l3}
Let a normal subgroup   $N$ of a group  $G$ be a direct product of isomorphic simple non-abelian groups. Then  $N$ is a direct product of minimal normal subgroups of $G$.
\end{lemma}

\begin{lemma}\label{(a)} Let $\mathfrak{F}$ be a hereditary saturated formation. Then $\mathrm{Z}_{\mathfrak{F}^*}(G)\leq\mathrm{Int}_{\mathfrak{F}^*}(G)$.  \end{lemma}

\begin{proof}

Let $\mathfrak{F}$ be a hereditary saturated formation with the canonical local definition $F$,
$M$ be a    $\mathfrak{F}^*$-maximal subgroup of $G$ and $N=M\mathrm{Z}_{\mathfrak{F}^*}(G)$. Let show that $N\in \mathfrak{F}^*$. It is sufficient to show that for every chief factor $H/K$ of $N$ below $\mathrm{Z}_{\mathfrak{F}^*}(G)$ either $H/K$ is $\mathfrak{F}$-central in $N$ or every $x\in N$ induces an inner automorphism on $H/K$. Let $1=Z_0\triangleleft Z_1\triangleleft\dots\triangleleft Z_n=\mathrm{Z}_{\mathfrak{F}^*}(G)$ be  a chief series of $G$ below $\mathrm{Z}_{\mathfrak{F}^*}(G)$. Then we may assume that $Z_{i-1}\leq K\leq H\leq Z_i$ for some $i$ by Jordan-H\"{o}lder theorem.

If $Z_i/Z_{i-1}$ is an $\mathfrak{F}$-central chief factor of $G$ then $G/C_G(Z_i/Z_{i-1})\in F(p)$ for all $p\in\pi(F(p))$ by Lemma \ref{l1}. Since $F(p)$ is hereditary by \cite[IV, 3.16]{s8}, $NC_G(Z_i/Z_{i-1})/C_G(Z_i/Z_{i-1})\simeq N/C_N(Z_i/Z_{i-1})\in F(p)$ for all $p\in\pi(F(p))$. Note that $N/C_N(H/K)$  is a quotient group of $N/C_N(Z_i/Z_{i-1})$. Thus $H/K$ is  an $\mathfrak{F}$-central chief factor of $N$  by Lemma \ref{l1}.

If $Z_i/Z_{i-1}$ is a $\mathfrak{F}$-eccentric chief factor of $G$ then every element of $G$ induced an inner automorphism on it. It means that  $Z_i/Z_{i-1}$ is a simple group. Hence it is also a chief factor of $N$. From $N\leq G$ it follows that every element of $N$ induces an inner automorphism on $Z_i/Z_{i-1}$.

Thus $N\in\mathfrak{F}^*$. So $N=M\mathrm{Z}_{\mathfrak{F}^*}(G)=M$. Therefore $\mathrm{Z}_{\mathfrak{F}^*}(G)\leq M$ for every $\mathfrak{F}^*$-maximal subgroup $M$ of $G$.
\end{proof}

\section{Proof of Theorem \ref{t1}}

Assume that $\mathfrak{F}$ is a hereditary saturated formation containing all nilpotent groups with the canonical local definition $F$. Then $F(p)$ is a hereditary formation for all primes $p$ (see \cite [IV, 3.14]{s8}).

Suppose that   $\mathrm{Int}_\mathfrak{F}(G)=\mathrm{Z}_\mathfrak{F}(G)$ holds for every group $G$.    Let show that  $\mathrm{Int}_{\mathfrak{F}^*}(G)=\mathrm{Z}_{\mathfrak{F}^*}(G)$ also holds for every group $G$.


Let $D$ be the intersection of all $\mathfrak{F}^*$-maximal   subgroups of $G$ and let $H/K$ be a chief factor of $G$ below $D$.

$(1)$ \emph{If $H/K$ is abelian then $MC_G(H/K)/C_G(H/K) \in F(p)$ for every $\mathfrak{F}^*$-maximal subgroup $M$ of $G$.}

If $H/K$ is abelian then it is an elementary abelian $p$-group for some prime $p$ and $H/K\in\mathfrak{F}$. Let $M$ be a $\mathfrak{F}^*$-maximal subgroup of $G$ and $K=H_0\triangleleft H_1\triangleleft\dots\triangleleft H_n=H$ be a part of a chief series of $M$. If $H_i/H_{i-1}$ is $\mathfrak{F}$-eccentric for some $i$ then every element of $M$ induces an inner automorphism on $H_i/H_{i-1}$. So $M/C_M(H_i/H_{i-1})\simeq 1\in F(p)$. Therefore $H_i/H_{i-1}$ is an $\mathfrak{F}$-central chief factor of $M$, a contradiction. Hence $H_i/H_{i-1}$ is an $\mathfrak{F}$-central chief factor of $M$ for all $i=1,\dots, n$. So $M/C_M(H_i/H_{i-1})\in F(p)$ by Lemma \ref{l1} for all $i=1,\dots, n$.
Therefore $M/C_M(H/K)\in\mathfrak{N}_pF(p)=F(p)$ by \cite[Lemma 1]{lvv}. Now $MC_G(H/K)/C_G(H/K)\simeq M/C_M(H/K)\in F(p)$ for every $\mathfrak{F}^*$-maximal subgroup $M$ of $G$.

$(2)$ \emph{If $H/K\in\mathfrak{F}$ is non-abelian then $MC_G(H/K)/C_G(H/K) \in F(p)$ for every $p\in\pi(H/K)$ and every $\mathfrak{F}^*$-maximal subgroup $M$ of $G$.}

If $H/K\in\mathfrak{F}$ is non-abelian then it is a direct product of isomorphic non-abelian simple $\mathfrak{F}$-groups. Let $M$ be a $\mathfrak{F}^*$-maximal subgroup of $G$. By Lemma \ref{l3} $H/K=H_1/K\times\dots\times H_n/K$ is a direct product of minimal normal subgroups $H_i/K$  of $M/K$. From $H_i/K\in \mathfrak{F}$ it follows that $(H_i/K)/\mathrm{O}_{p', p}(H_i/K)\simeq H_i/K\in F(p)$ for all $p\in\pi(H_i/K)$ and all $i=1,\dots, n$.
Assume that $H_i/K$ is a $\mathfrak{F}$-eccentric chief factor of $M/K$  for  some $i$. It means that every element of $M$ induces an inner automorphism on $H_i/K$. So $M/C_M(H_i/K)\simeq H_i/K\in F(p)$, a contradiction.

Now $H_i/K$ is $\mathfrak{F}$-central in $M/K$ for all $i=1,\dots, n$. Therefore $M/C_M(H_i/K)\in F(p)$ for all $p\in\pi(H_i/K)$ by Lemma \ref{l1}. Note that $C_M(H/K)=\cap_{i=1}^n C_M(H_i/K)$. Since $F(p)$ is a formation,    $M/\cap_{i=1}^n C_M(H_i/K)= M/C_M(H/K)\in F(p)$ for all $p\in\pi(H/K)$. It means that $MC_G(H/K)/C_G(H/K)\simeq M/C_M(H/K)\in F(p)$ for every $p\in\pi(H/K)$ and every $\mathfrak{F}^*$-maximal subgroup $M$ of $G$.

$(3)$ \emph{If $H/K\in\mathfrak{F}$ then all $\mathfrak{F}$-subgroups of $G/C_G(H/K)$ are $F(p)$-groups.}

Let $Q/C_G(H/K)$ be a  $\mathfrak{F}$-maximal subgroup of $G/C_G(H/K)$. Then there exists a  $\mathfrak{F}$-maximal subgroup $N$ of $G$ with $NC_G(H/K)/C_G(H/K)=Q/C_G(H/K)$ by \cite[1,  5.7]{s5}. From $\mathfrak{F}\subseteq\mathfrak{F}^*$ it follows that there exists a $\mathfrak{F}^*$-maximal subgroup $L$ of $G$ with $N\leq L$. So  $Q/C_G(H/K)\leq LC_G(H/K)/C_G(H/K)\in F(p)$ by $(1)$ and $(2)$. Since $F(p)$ is hereditary,  $Q/C_G(H/K)\in F(p)$. It means that all $\mathfrak{F}$-maximal subgroups of $G/C_G(H/K)$ are $F(p)$-groups. Hence all $\mathfrak{F}$-subgroups of $G/C_G(H/K)$ are $F(p)$-groups.

$(4)$ \emph{If $H/K\in\mathfrak{F}$   then it is $\mathfrak{F}$-central in $G$.}

Assume now that $H/K$ is not  $\mathfrak{F}$-central in $G$. So $G/C_G(H/K)\not\in F(p)$ for some $p\in\pi(H/K)$ by Lemma \ref{l1}. It means that    $G/C_G(H/K)$ contains a $s$-critical for $F(p)$ subgroup  $S/C_G(H/K)$ for some $p\in\pi(H/K)$. Since   $\mathrm{Int}_\mathfrak{F}(G)=\mathrm{Z}_\mathfrak{F}(G)$ holds for every group $G$,    $S/C_G(H/K)\in\mathfrak{F}$ by Theorem \ref{t0}. Therefore $S/C_G(H/K)\in F(p)$ by $(4)$, the contradiction. Thus $H/K$ is    $\mathfrak{F}$-central in $G$.

$(5)$ \emph{If $H/K\not\in\mathfrak{F}$ is non-abelian then every element of $G$ induces an inner automorphism on it.}

    Let $M$ be a $\mathfrak{F}^*$-maximal subgroup of $G$. By Lemma \ref{l3} $H/K=H_1/K\times\dots\times H_n/K$ is a direct product of minimal normal subgroups $H_i/K$  of $M/K$. Since $H_i/K\not\in\mathfrak{F}$ for all $i=1,\dots, n$, it is an $\mathfrak{F}$-eccentric chief factor of $M$ for all $i=1,\dots, n$. So every element of $M$ induces an inner automorphism on $H_i/K$ for all $i=1,\dots, n$. It means that for all  $a\in M$ and for all $i=1,\dots, n$ there is  $x(a, i)\in H_i/K$ and  $(H_i/K)^a=(H_i/K)^{x(a, i)}$. So for all $a\in M$ there is $x(a)=x(a, 1)\dots x(a, n)\in H/K$  such that $(H/K)^a=(H_1/K)^a\times\dots\times (H_n/K)^a=(H_1/K)^{x(a, 1)}\times\dots\times (H_n/K)^{x(a, n)}=(H_1/K)^{x(a)}\times\dots\times (H_n/K)^{x(a)}=(H/K)^{x(a)}$.

    It means that for every $\mathfrak{F}^*$-maximal subgroup $M$ of $G$ every element of $M$ induces an inner automorphism on $H/K$. Since $\mathfrak{N}\subseteq\mathfrak{F}$, $\langle x\rangle\in\mathfrak{F}$ for every $x\in G$. From $\mathfrak{F}\subseteq\mathfrak{F}^*$ it follows that for every element $x$  of $G$ there is a $\mathfrak{F}^*$-maximal subgroup $M$ of $G$ with $x\in M$. Thus  every element   of $G$ induces an inner automorphism on $H/K$.

$(7)$ \emph{The final step.}

If $H/K\in\mathfrak{F}$ then from $\mathfrak{F}\subseteq\mathfrak{F}^*$ and $(4)$ it follows that $H/K$  is $\mathfrak{F}^*$-central in $G$. Assume that $H/K\not\in\mathfrak{F}$. By $(5)$ every element of $G$ induces an inner automorphism on it. Hence $H/K$ is a simple non-abelian group.  Since $G/C_G(H/K)$ is isomorphic to some subgroup of $Inn(H/K)$, we see that $G/C_G(H/K)\simeq H/K$. It is straightforward to check that $H/K\leftthreetimes G/C_G(H/K)\simeq H/K\times H/K$. From $H/K\in\mathfrak{F}^*$ it follows that $H/K\leftthreetimes G/C_G(H/K)\in\mathfrak{F}^*$. Hence $H/K$  is $\mathfrak{F}^*$-central in $G$.

Thus every chief factor of $G$ below $D$ is  $\mathfrak{F}^*$-central in $G$. Hence $D\leq \mathrm{Z}_{\mathfrak{F}^*}(G)$. According to Lemma \ref{(a)} $\mathrm{Z}_{\mathfrak{F}^*}(G)\leq D$. Therefore $\mathrm{Z}_{\mathfrak{F}^*}(G)=D$.

\medskip

 Suppose that   $\mathrm{Int}_{\mathfrak{F}^*}(G)=\mathrm{Z}_{\mathfrak{F}^*}(G)$ holds for every group $G$.      Let show that  $\mathrm{Int}_\mathfrak{F}(G)=\mathrm{Z}_\mathfrak{F}(G)$ also holds for every group $G$.   Assume the contrary. Then there exists a $s$-critical for $F(p)$ group $G\not\in\mathfrak{F}$ for some $p\in\mathbb{P}$ by Theorem \ref{t0}. We may assume that $G$ is a minimal group with this property. Then $\mathrm{O}_p(G)=\Phi(G)=1$ and $G$ has an unique minimal normal subgroup  by \cite[Lemma 2.10]{h4}. Note that $G$ is also  $s$-critical for $\mathfrak{F}$.

Assume that $G\not\in\mathfrak{F}^*$. Then there exists a simple $\mathbb{F}_pG$-module $V$ which is faithful for $G$ by \cite[10.3B]{s8}. Let $T=V\leftthreetimes G$. Note that $T\not\in\mathfrak{F}^*$.  Let $M$ be a maximal subgroup of $T$. If $V\leq M$ then $M=M\cap VG=V(M\cap G)$ where $M\cap G$ is a maximal subgroup of $G$. From $M\cap G\in F(p)$ and $F(p)=\mathfrak{N}_pF(p)$ it follows that $V(M\cap G)=M\in F(p)\subseteq\mathfrak{F}\subseteq\mathfrak{F}^*$. Hence $M$ is an $\mathfrak{F}^*$-maximal   subgroup of $G$. If $V\not\leq M$ then $M\simeq T/V\simeq G\not\in\mathfrak{F}$. Now it is clear that the sets of all maximal $\mathfrak{F}$-subgroups and all $\mathfrak{F}^*$-maximal subgroups of $T$ coincide. Therefore $V$ is the intersection of all   $\mathfrak{F}^*$-maximal subgroups of $T$. From $T\simeq V\leftthreetimes T/C_T(V)\not\in\mathfrak{F}^*$ it follows that $V\not\leq\mathrm{Z}_{\mathfrak{F}^*}(T)$, a contradiction.

Assume that $G\in\mathfrak{F}^*$. Let $N$ be a minimal normal subgroup of $G$. If $N<G$ then $N\in\mathfrak{F}$. As follows from $(4)$ $N$ is an $\mathfrak{F}$-central chief factor of $G$. So $N\leq\mathrm{Z}_{\mathfrak{F}}(G)$. Since $N$ is an unique minimal normal subgroup of       $s$-critical for  $\mathfrak{F}$-group $G$ and $\Phi(G)=1$, we see that $G/N\in\mathfrak{F}$. Hence $G\in\mathfrak{F}$, a contradiction. Thus $N=G$ is a simple group. From $\mathfrak{N}\subseteq\mathfrak{F}$ it follows that $G$ is non-abelian. 

Let $p\in\pi(G)$. According to \cite{20} there is a
 simple Frattini $\mathbb{F}_pG$-module $A$ which is faithful for $G$. By known Gasch\"{u}tz
 theorem \cite{41}, there
exists a Frattini extension  $A\rightarrowtail R\twoheadrightarrow G$
such that $A\stackrel {G}{\simeq} \Phi(R)$ and $R/\Phi(R)\simeq G$. Since $\Phi(R)\leftthreetimes R/C_R(\Phi(R))\simeq A\leftthreetimes G\not\in\mathfrak{F}$, we see that $R\not\in\mathfrak{F}^*$. Let $M$ be a maximal subgroup of $R$. Then $M/\Phi(R)$ is isomorphic to a maximal subgroup of $G$. So  $M/\Phi(R)\in F(p)$. From $\mathfrak{N}_pF(p)=F(p)$ it follows  that $M\in F(p)\subseteq\mathfrak{F}\subseteq\mathfrak{F}^*$. Hence the sets of maximal and $\mathfrak{F}^*$-maximal subgroups of $R$ coincide.  Thus $\Phi(R)=\mathrm{Z}_{\mathfrak{F}^*}(R)$. From $R/\mathrm{Z}_{\mathfrak{F}^*}(R)\simeq G\in\mathfrak{F}^*$ it follows that $R\in\mathfrak{F}^*$, the final contradiction.


\begin{thebibliography}{20}

\leftskip=-7mm
\itemindent=10mm
\parskip=-0mm
\parsep=0mm
\itemsep=0mm
\labelwidth=-5mm




\bibitem{h1}
R. Baer.  Group Elements of Prime Power Index. \emph{Trans. Amer. Math Soc.}  \textbf{75}(1)  (1953),  20--47.


     \bibitem{s9}
A. Ballester-Bolinches  and L.\,M. Ezquerro.  \emph{Classes of Finite Groups} (Springer, 2006).


\bibitem{s8}  K. Doerk and  T. Hawkes. \emph{Finite soluble groups} (Walter de Gruyter, 1992).

\bibitem{41}
W. Gash\"{u}tz.  \"{U}ber modulare Darstellungen endlicher Gruppen, die von freien Gruppen induziert werden. \emph{Math. Z.} \textbf{60} (1954), 274--286.



\bibitem{s5} W. Guo. \emph{Structure theory for canonical classes of finite groups} (Springer, 2015).

\bibitem{sk2}   W. Guo and A.\,N. Skiba. On  finite quasi-$\mathfrak{F}$-groups. \emph{Comm. Algebra} \textbf{37} (2009), 470–481.

\bibitem{sk1} W. Guo and A.\,N. Skiba. On some classes of finite quasi-$\mathfrak{F}$-groups. \emph{J. Group Theory} \textbf{12} (2009), 407–417.


\bibitem{Hup}
B. Huppert and N. Blackburn. \emph{Finite groups}, vol. 3  (Springer-Verlag,   1982).

\bibitem{jm}
J. Moori. On the Automorphism Group of the Group $D_4(2)$. \emph{J. Algebra} \textbf{80} (1983), 216--225.

\bibitem{20}
P. Schmid, R.\,L. Griess. The Frattini module.  \emph{Arch. Math.}  \textbf{30} (1978), 256--266.


\bibitem{s6}
L.\,A. Shemetkov and A.\,N. Skiba.  \emph{Formations of algebraic systems}.  (Nauka, Moscow, 1989). (In Russian)

\bibitem{h4} A.\,N. Skiba.  On the $\mathfrak{F}$-hypercenter and the intersection of all $\mathfrak{F}$-maximal subgroups of a finite group. \emph{Journal of Pure and Aplied Algebra} \textbf{216}(4) (2012), 789--799.


\bibitem{lvv} A.\,F. Vasil’ev and T.\,I. Vasilyeva. On finite groups whose principal factors are simple groups. \emph{Russian Mathematics} (Izvestiya VUZ. Matematika)  \textbf{41}(11) (1997),   8--12.


\bibitem{Ved} V. A. Vedernikov. On some classes of finite groups. \emph{Dokl. Akad. Nauk
BSSR} \textbf{32}(10) (1988), 872-875. (In Russian)





























\end{thebibliography}
\end{document}